%% file: main.tex
\documentclass[leqn,10pt]{amsart}
\usepackage[utf8]{inputenc}
\usepackage{a4}
\usepackage[safeinputenc=true,style=alphabetic,firstinits,maxnames=4,useprefix=true,backend=bibtex]{biblatex}

\bibliography{fluctuations}
\usepackage{amsmath,amssymb,amscd,bbm}
\usepackage{amsthm}
\usepackage{enumerate}
\usepackage[mathscr]{eucal}
\usepackage[all]{xy}

\newtheorem{thm}{Theorem}[section]
\newtheorem*{thmn}{Theorem}
\newtheorem{lemma}[thm]{Lemma}
\newtheorem{prop}[thm]{Proposition}
\newtheorem{cor}[thm]{Corollary}
\newtheorem{claimn}{Claim}

\theoremstyle{definition}

\newtheorem{exa}[thm]{Example}

\theoremstyle{remark}

\newtheorem{rem}[thm]{Remark}

\numberwithin{equation}{section}

\input{makro}

\newcommand{\cntm}[1]{\left|#1\right|}
\newcommand{\Rp}{\R_{\geq 0}}
\newcommand{\Rps}{\R_{>0}}
\newcommand{\Zp}{\Z_{\geq 0}}
\newcommand{\bl}{\mathrm B}
\newcommand{\bdr}[1]{\partial_{#1}}
\newcommand{\intr}[1]{\mathrm{int}_{#1}}

\newcommand{\avg}[1]{\bbE_{#1}}
\newcommand{\indi}[1]{\mathbf{1}_{#1}}

\date{\today}

\begin{document}

\title[Fluctuations of ergodic averages]{Fluctuations of ergodic averages\\ for
  Actions of Groups of Polynomial Growth.}

\author[Nikita Moriakov]{Nikita Moriakov}
\address{Delft Institute of Applied Mathematics, Delft University of Technology,
P.O. Box 5031, 2600 GA Delft, The Netherlands}

\email{n.moriakov@tudelft.nl}

\subjclass{Primary  28D05, 28D15}
\renewcommand{\subjclassname}{\textup{2000} Mathematics Subject
    Classification}

\date{\today}

\begin{abstract}
It was shown by S. Kalikow and B. Weiss that, given a measure-preserving action
of $\Z^d$ on a probability space $\mathrm{X}$ and a nonnegative
measurable function $f$ on $\mathrm{X}$, the probability that the sequence of
ergodic averages
$$
\frac 1 {(2k+1)^d} \sum\limits_{g \in [-k,\dots,k]^d} f(g \cdot x)
$$
has at least $n$ fluctuations across an interval $(\alpha,\beta)$
can be bounded from above by $c_1 c_2^n$ for some universal constants
$c_1 \in \mathbb{R}$ and $c_2 \in (0,1)$, which depend only on $d,\alpha,\beta$. The purpose of this article is
to generalize this result to measure-preserving actions of groups of
polynomial growth. As the main tool we develop a generalization of effective
Vitali covering theorem for groups of polynomial growth.
\end{abstract}

\maketitle

\section{Introduction}
Given an integer $n \in \Zp$ and some numbers $\alpha,\beta \in \R$ such that $\alpha<\beta$, a sequence of real numbers $(a_i)_{i = 1}^k$ is said to \textbf{fluctuate at
least $n$ times} across the interval $(\alpha,\beta)$ if there are indexes $1 \leq i_0 < i_1 < \dots < i_n
\leq k$ such
that
\begin{aufziii}
\item if $j$ is odd, then $a_{i_j} < \alpha$;
\item if $j$ is even, then $a_{i_j} > \beta$.
\end{aufziii} 
In this case it is clear that for every even $j$ we have
\[
a_{i_j} > \beta \quad \text{and} \quad a_{i_{j+1}} < \alpha,
\]
i.e., $(a_i)_{i=1}^k$ has at least $\lceil \frac n 2 \rceil$
\textbf{downcrossings} from $\beta$ to $\alpha$ and at least $\lfloor
\frac n 2 \rfloor$ \textbf{upcrossings} from $\alpha$ to $\beta$. If
$(a_i)_{i \geq 1}$ is an infinite sequence of real numbers, we use the
same terminology and say
that $(a_i)_{i \geq 1}$ fluctuates at least $n$ times across the
interval $(\alpha,\beta)$ if some initial segment $(a_i)_{i=1}^k$ of
the sequence fluctuates at least $n$ times across $(\alpha,\beta)$. We
denote the sets of all real-valued sequences having at least $n$
fluctuations across an interval $(\alpha, \beta)$ by
$\calF_{(\alpha,\beta)}^n$, and it will be clear from the context if
we are talking about finite or infinite sequences.

The main result of this article is the following theorem, which
generalizes the results in \cite{kw1999} about fluctuations of
averages of nonnegative functions.
\begin{thmn}
Let $\Gamma$ be a group of polynomial growth and
let $(\alpha, \beta) \subset \Rps$ be some nonempty interval. Then
there are some constants $c_1,c_2 \in \Rps$ with $c_2<1$, which depend
only on $\Gamma$, $\alpha$
and $\beta$, such that the following assertion holds. 

For any probability space $\prX=(X, \calB, \mu)$, any
measure-preserving action of $\Gamma$ on $\prX$ and any measurable $f
\geq 0$ on $X$ we have
\[
\mu(\{ x: (\avg{g \in \bl(k)} f(g \cdot x))_{k \geq 1} \in \calF_{(\alpha,\beta)}^N  \})
< c_1 c_2^N
\]
for all $N \geq 1$.
\end{thmn}

The paper is structured as follows. We provide some background on
groups of polynomial growth in Section \ref{ss.grouppolgr}, discuss
some special properties of averages on groups of polynomial growth and
a transference principle in Section \ref{ss.avgongrp} and prove
effective Vitali covering theorem in Section \ref{ss.vitcov}. The
main theorem of this paper is Theorem \ref{t.expdec}, which is proved
in Section \ref{s.upcrineq}.

This research was done during the author's PhD studies under the
supervision of Markus Haase. I would like to thank him for his support
and advice.

\section{Preliminaries}
\subsection{Groups of Polynomial Growth}
\label{ss.grouppolgr}
Let $\Gamma$ be a finitely generated group and $\{
\gamma_1,\dots,\gamma_k\}$ be a fixed generating set. Each element
$\gamma \in \Gamma$ can be represented as a product
$\gamma_{i_1}^{p_1} \gamma_{i_2}^{p_2} \dots \gamma_{i_l}^{p_l}$ for
some indexes $i_1,i_2,\dots,i_l \in 1,\dots,k$ and some integers $p_1,p_2,\dots,p_l \in \Z$. We define the \textbf{norm} of an element $\gamma
\in \Gamma$ by
\[
\| \gamma \|:=\inf\{ \sum\limits_{i=1}^l |p_i|: \gamma =
\gamma_{i_1}^{p_1} \gamma_{i_2}^{p_2} \dots \gamma_{i_l}^{p_l} \},
\]
where the infinum is taken over all representations of $\gamma$ as a
product of the generating elements. The norm $\| \cdot \|$ on $\Gamma$, in general, does depend on the generating
set. However, it is easy to show \cite[Corollary 6.4.2]{ceccherini2010} that
two different generating sets produce equivalent norms. We will always
say what generating set is used in the definition of a norm, but we will
omit an explicit reference to the generating set later on. For every
$n \in \Rp$ let
\[
\bl(n):= \{ \gamma \in \Gamma: \| \gamma \| \leq n\} 
\]
be the closed ball of radius $n$.

The norm $\| \cdot \|$ yields a right invariant metric on $\Gamma$
defined by
\[
d_R(x,y):=\| x y^{-1}\| \quad (x,y \in \Gamma),
\]
and a left invariant metric on $\Gamma$ defined by
\[
d_L(x,y):=\| x^{-1} y\| \quad (x,y \in \Gamma),
\]
which we call the \textbf{word metrics}. The right invariance of $d_R$ means that
the right multiplication
\[
R_g: \Gamma \to \Gamma, \quad x \mapsto x g \quad ( x \in \Gamma)
\]
is an isometry for every $g \in \Gamma$ with respect to
$d_R$. Similarly, the left invariance of $d_L$ means that the left
multiplications are isometries with respect to $d_L$.
We let $d:=d_R$ and view
$\Gamma$ as a metric space with the metric $d$. For $x\in \Gamma$, $r
\in \Rp$ let
\[
\bl(x,r):=\{ y \in \Gamma: d(x,y) \leq r\}
\]
be the closed ball of radius $r$ with center $x$. Using the right
invariance of the metric $d$, it is easy to see that
\[
\cntm{\bl(x,r)} = \cntm{\bl(y,r)} \quad \text{ for all } x,y \in \Gamma.
\]
Let $\ue \in \Gamma$ be the neutral element. It is clear that 
\[
\bl(n) = \{ \gamma: d_R(\ue,\gamma) \leq n\} = \{ \gamma:
d_L(\ue,\gamma) \leq n\},
\]
i.e., the ball $\bl(n)$ is precisely the ball $\bl(\ue, n)$ with respect to the
left and the right word metric.

It is important to understand how fast the balls $\bl(n)$ in the group
$\Gamma$ grow as $n
\to \infty$. The \textbf{growth function} $\gamma: \N \to \N$ is defined by
\[
\gamma(n):=\cntm{\bl(n)} \quad (n \in \N).
\]
We say that the group $\Gamma$ is of
\textbf{polynomial growth} if there are constants $C,d>0$ such that
for all $n \geq 1$ we have
\[
\gamma(n) \leq C(n^d+1).
\]

\begin{exa}
\label{ex.zdex}
Consider the group $\Z^d$ for $d \in \N$ and let $\gamma_1,\dots,\gamma_d \in \Z^d$ be the
standard basis elements of $\Z^d$. That is, $\gamma_i$ is defined by
\[
\gamma_i(j):=\delta_i^j \quad (j=1,\dots, d)
\] 
for all $i=1,\dots,d$. We consider the generating set given by elements $\sum\limits_{k \in I} (-1)^{\varepsilon_k}\gamma_k$ for all
subsets $I \subseteq [1,d]$ and all functions $\varepsilon_{\cdot} \in
\{ 0,1\}^I$. Then it is easy
to see by induction on dimension that $\bl(n) = [-n,\dots,n]^d$, hence
\[
\cntm{\bl(n)} = (2n+1)^d \quad \text{ for all } n \in \N
\]
with respect to this generating set, i.e., $\Z^d$ is a group of polynomial growth.
\end{exa}

Let $d \in \Zp$. We say that the group $\Gamma$ has \textbf{polynomial growth
of degree $d$} if there is a constant $C>0$ such that
\[
\frac 1 C n^d \leq \gamma(n) \leq C n^d \quad \text{ for all } n \in \N.
\]
It was shown in \cite{bass1972} that, if $\Gamma$ is a finitely
generated nilpotent group, then $\Gamma$ has polynomial growth of some
degree $d \in \Zp$. Furthermore, one can show \cite[Proposition
6.6.6]{ceccherini2010} that if $\Gamma$ is a group and $\Gamma' \leq \Gamma$
is a finite index, finitely generated nilpotent subgroup, having
polynomial growth of degree $d \in \Zp$, then the group $\Gamma$ has
polynomial growth of degree $d$ as well. A surprising fact is that the
converse is true as well. Namely, it was proved in \cite{gromov1981}
that, if $\Gamma$ is a group of polynomial growth, then there is a
finite index, finitely generated nilpotent subgroup $\Gamma' \leq
\Gamma$. It follows that if $\Gamma$ is a group of polynomial growth
with the growth function $\gamma$, then there is a constant $C>0$ and
an integer $d\in \Zp$, called the \textbf{degree of polynomial growth}, such that
\[
\frac 1 C n^d \leq \gamma(n) \leq C n^d \quad \text{ for all } n \in \N.
\]
An even stronger result was obtained in \cite{pansu1983}, where it is
shown that, if $\Gamma$ is a group of polynomial growth of degree $d
\in \Zp$, then the limit
\begin{equation}
\label{eq.pansu}
c_{\Gamma}:=\lim\limits_{n \to \infty} \frac{\gamma(n)}{n^d}
\end{equation}
exists. As a consequence, one can show that groups of polynomial growth are amenable. 
\begin{prop}
Let $\Gamma$ be a group of polynomial growth. Then $(\bl(n))_{n \geq
1}$ is a F{\o}lner sequence in $\Gamma$.
\end{prop}
\begin{proof}
We want to show that for every $g \in \Gamma$
\[
\lim\limits_{n \to \infty} \frac{\cntm{g \bl(n) \sdif
    \bl(n)}}{\cntm{\bl(n)}} = 0.
\]
Let $m:=d(g,e) \in \Zp$. Then $g \bl(n) \subseteq \bl(n+m)$, hence
\[
\frac{\cntm{g \bl(n) \sdif
    \bl(n)}}{\cntm{\bl(n)}} \leq \frac{\cntm{\bl(n+m)} - \cntm{ \bl(n)}} {\cntm{\bl(n)}} \to 0,
\]
where we use the existence of the limit in Equation \eqref{eq.pansu}.
\end{proof}

It will be useful later to have a special notion for the points which
are `close enough' to the boundary of a ball in $\Gamma$. Let $W:=\bl(y,s)$ be some ball in
$\Gamma$. For a given $r \in \Rps$ the \textbf{$r$-interior}
of $W$ is defined as
\[
\intr{r}(W):=\bl(y,(1-5/r)s).
\] 
The \textbf{$r$-boundary} of $W$ is defined as
\[
\bdr{r}(W):=W \setminus \intr{r}(W).
\]
If a set $\calC$ is a disjoint collection of balls in $\Gamma$, we define the
$r$-interior and the $r$-boundary of $\calC$ as
\[
\intr{r}(\calC):=\bigsqcup\limits_{W \in \calC} \intr{r}(W)
\]
and
\[
\bdr{r}(\calC):=\bigsqcup\limits_{W \in \calC} \bdr{r}(W)
\]
respectively. It will be essential to know that the $r$-boundary
becomes small (respectively, the $r$-interior becomes large) for large
enough balls and large enough $r$. More precisely, we state the
following lemma, whose proof follows from the result of Pansu (see Equation
\eqref{eq.pansu}).

\begin{lemma}
\label{l.smallbdr}
Let $\Gamma$ be a group of polynomial growth and $\delta \in (0,1)$ be
some constant. Then
there exist constants $n_0, r_0 \in \N$, depending only on $\Gamma$
and $\delta$, such that the following holds. If $\calC$ is a finite collection of disjoint balls with radii greater than $n_0$,
then for all $r > r_0$
\[
\cntm{\intr{r} ( \calC )} > (1-\delta) \cntm{\bigsqcup\limits_{W \in \calC} W}
\]
and
\[
\cntm{\bdr{r} ( \calC )} < \delta \cntm{\bigsqcup\limits_{W \in \calC} W}.
\]
\end{lemma}

\subsection{Averages on Groups of Polynomial Growth and a Transference Principle}
\label{ss.avgongrp}

We collect some useful results about averages on groups of polynomial
growth in this subsection. At the end of the subsection we will
discuss a transference principle, which will become essential later in Section
\ref{s.upcrineq}. We start with a
preliminary lemma, whose proof is straightforward.

\begin{lemma}
\label{l.ballgrowth}
Let $f$ be a nonnegative function on a group of polynomial growth $\Gamma$. Let $\{ B_1, \dots,
B_k\}$ be some disjoint balls in $\Gamma$ such that
\[
\avg{g \in B_i} f(g) > \beta \quad \text{for each } i=1,\dots,k.
\]
Let $B$ be a ball in $\Gamma$, containing all $B_i$'s, such that
\begin{equation*}
\avg{g \in B} f(g) < \alpha.
\end{equation*}
Then
\[
\frac{\sum\limits_{i=1}^k \cntm{B_i}}{\cntm{B}} < \frac{\alpha}{\beta}.
\]
\end{lemma}
\noindent
We refine this result as follows.
\begin{lemma}
\label{l.uskip}
Let $\varepsilon \in (0,1)$. There is $n_0 \in \N$, depending only on
the group of polynomial growth $\Gamma$ and $\varepsilon$, such that
the following assertion holds. Given a nonnegative function $f$
on $\Gamma$, the condition
\begin{equation}
\label{eq.fluctcond}
\avg{g \in \bl(n)} f( g ) > \beta \quad \text{and} \quad
\avg{g \in \bl(m)} f( g ) < \alpha
\end{equation}
for some $n_0 \leq n < m$  and an interval $(\alpha,\beta) \subset \Rps$ implies that
\[
\frac m n > (1-\varepsilon) \left(\frac {\beta}{\alpha} \right)^{1/d}.
\]
\end{lemma}
\begin{proof}
First of all, note that condition \eqref{eq.fluctcond} implies that
\[
\frac{\cntm{\bl(m)}}{\cntm{\bl(n)}} > \frac{\beta}{\alpha}
\]
for \emph{all} indexes $n<m$ (see the previous lemma). Using the result of Pansu (Equation
\eqref{eq.pansu}), we deduce that there is $n_0$ depending only on
$\Gamma$ and $\varepsilon$ such that for all
$n_0 \leq n<m$ we have
\[
\frac{m^d}{n^d}>(1-\varepsilon)^d \frac{\cntm{\bl(m)}}{\cntm{\bl(n)}}.
\]
This implies that
\[
\frac m n > (1-\varepsilon) \left(  \frac{\beta}{\alpha}\right)^{1/d},
\]
and the proof of the lemma is complete.
\end{proof}
\noindent
Lemma \ref{l.uskip} has the following straightforward corollary.
\begin{cor}
\label{c.skip}
For a constant $\varepsilon \in (0,1)$ and a group of polynomial growth $\Gamma$ let  $n_0:=n_0(\varepsilon)$ be given
by Lemma \ref{l.uskip}. Given a measure-preserving action of $\Gamma$
on a probability space $\prX$, a nonnegative function $f$ on $X$
and $x \in X$, the condition that the sequence
\[
\left( \avg{g \in \bl(i)} f(g \cdot x) \right)_{i=n}^m
\] 
fluctuates at least $k$ times across an interval $(\alpha, \beta)
\subset \Rps$ with $n>n_0$
implies that
\[
\frac m n > (1-\varepsilon)^{\lceil \frac k 2 \rceil} \left(
  \frac{\beta}{\alpha} \right)^{{\lceil \frac k 2 \rceil} \cdot \frac
  1 d}
\]
\end{cor}

Finally, we will need an adapted version of the `easy direction' in Calder\'{o}n's
transference principle for groups of polynomial growth. Suppose that a
group $\Gamma$ of polynomial growth acts on a probability space
$\prX=(X,\calB,\mu)$ by measure-preserving transformations and that we
want to estimate the size of a measurable set $E$. Fix an integer
$m \in \Zp$. For an integer $L \in \N$ and a point $x \in
X$ we define the set
\[
B_{L,m,x}:= \{ g: \ g \cdot x \in E \text{  and  } \| g \| \leq L-m \} \subseteq \bl(L).
\]
The lemma below tells us that each universal upper bound on the
density of $B_{L,m,x}$ in $\bl(L)$ bounds the measure of $E$ from
above as well.
\begin{lemma}[Transference principle]
\label{l.caldtrans}
Suppose that for a given constant $t \in \Rp$ the following holds:
there is some $L_0 \in \N$ such that for all $L \geq L_0$ and for $\mu$-almost all $x \in
X$ we have
\[
\frac 1 {\cntm{\bl(L)}} \cntm{B_{L,m,x}} \leq t.
\]
Then
\[
\mu(E) \leq t.
\]
\end{lemma}
\begin{proof}
Indeed, since $\Gamma$ acts on $\prX$ by measure-preserving
transformations, we have
\[
\sum\limits_{g \in \bl(L)} \int\limits_{\prX} \indi{E}(g \cdot x) d \mu = \cntm{\bl(L)} \mu(E).
\]
Then
\begin{align*}
\mu(E) &= \int\limits_{\prX} \left( \frac 1 {\cntm{\bl(L)}} \sum\limits_{g \in
      \bl(L)} \indi{E} (g \cdot x) \right) d \mu \leq \\
&\leq \int\limits_{\prX} \left( \frac{\cntm{B_{L,m,x}} + \cntm{\bl(L)
  \setminus \bl(L-m)}}{\cntm{\bl(L)}}  \right) d \mu,
\end{align*}
and the proof is complete since $L$ can be arbitrarily large and $\Gamma$ is a
group of polynomial growth.
\end{proof}

\subsection{Vitali Covering Lemma}
\label{ss.vitcov}
In this section we discuss the generalization of Effective Vitali
Covering lemma from \cite{kw1999} to groups of polynomial growth. We fix some notation
first. Given a number $t \in \Rp$ and a ball $B=\bl(x,r) \subseteq
\prX$ in a metric space $\prX$, we denote by $t \cdot B$ the
$t$-enlargement of $B$, i.e., the ball $\bl(x,rt)$. We state the
basic finitary Vitali covering lemma first, whose proof is well-known.

\begin{lemma}
\label{l.fvc}
Let $\calB:=\{ B_1,\dots,B_n \}$ be a finite collection of balls in a metric
space $\prX$. Then there is a finite subset $\{ B_{j_1},\dots,
B_{j_m}\} \subseteq \calB$ consisting of pairwise disjoint balls such that
\[
\bigcup\limits_{i=1}^n B_i \subseteq  \bigcup\limits_{l=1}^m 3 \cdot B_{j_l}.
\]
\end{lemma}

Infinite version of this lemma is used, for example, in the proof of the
standard Vitali
covering theorem, which can be generalized to arbitrary doubling
measure spaces. However, the standard Vitali covering theorem is not
sufficient for our purposes. It was shown in \cite{kw1999} that the groups $\Z^d$ for $d \in \N$, which are of course doubling measure spaces when endowed
with the counting measure and the word metric, enjoy a particularly
useful `effective' version of the theorem. We prove a generalization
of this result to groups of polynomial growth below. 

\begin{thm}[Effective Vitali covering]
\label{t.evc}
Let $\Gamma$ be a group of polynomial growth of degree $d$. Let $C \geq 1$ be a
constant such that
\[ 
\frac 1 C m^d \leq \gamma(m) \leq C m^d \quad \text{ for all } m \in \N
\]
and let $c:=3^d C^2$. Let $R,n,r>2$ be some
fixed natural numbers and $X \subseteq
\bl(R)$ be a subset of the ball $\bl(R) \subset \Gamma$. Suppose that
to each $p \in X$ there are associated balls $A_1(p),\dots,A_n(p)$
such that the following assertions hold:
\begin{aufzi}
\item $p \in A_i(p) \subseteq \bl(R)$ for $i=1,\dots,n$;
\item For all $i=1,\dots,n-1$ the $r$-enlargement of $A_i(p)$ is
  contained in $A_{i+1}(p)$.
\end{aufzi}
Let
\[
S_i:=\bigcup\limits_{p \in X} A_i(p) \quad (i=1,\dots,n).
\]
There is a disjoint subcollection $\calC $ of $\{ A_i(p) \}_{p \in X, i=1,\dots,n}$ such that the
following conclusions hold:
\begin{aufzi}
\item  The union of $\left( 1+ \frac 4 {r-2} \right)$-enlargements of
  balls in $\calC$ together with the the set $S_n \setminus S_1$
  covers all but at most $\left( \frac {c-1} c \right)^n$ of $S_n$;
\item The measure of the union of $\left( 1+ \frac 4 {r-2} \right)$-enlargements of
  balls in $\calC$ is at least $(1 - \left( \frac {c-1} c \right)^n)$
  times the measure of $S_1$.
\end{aufzi}
\end{thm}

\begin{rem}
\label{r.maxball}
Prior to proceeding to the proof of the theorem we make the following
remarks. Firstly, we do not require the balls $A_i(p)$ from the
theorem to be centered around $p$. Secondly, the balls of the form $A_i(p)$ for
$i=1,\dots,n$ and $p \in X$ will be called \textbf{$i$-th level
balls}. An $i$-th level ball $A_i(p)$ is called \textbf{maximal} if it
is not contained in any other $i$-th level ball. It is clear that each
$S_i$ is the union of maximal $i$-level balls as well. It will follow from the proof below that the balls in $\calC$ can be chosen to
be maximal.
\end{rem}

\begin{proof}
To simplify the notation, let
\[
s:=1+\frac 4 {r-2}
\]
be the scaling factor that is used in the theorem. The main idea of
the proof is to cover a positive fraction of $S_n$ by a disjoint union
of $n$-level balls via Lemma \ref{l.fvc}, then cover a positive
fraction of what remains in $S_{n-1}$ by a disjoint union of
$(n-1)$-level balls and so on. Thus we begin by covering a fraction of $S_n$ by $n$-level balls. Let $\calC_n \subseteq \{ A_n(p) \}_{p \in X}$ be the collection of
disjoint balls, obtained by applying Lemma \ref{l.fvc} to the
collection of all $n$-th level \emph{maximal} balls. For
every ball $B=\bl(p,m) \in \calC_n$ we have 
\[
\cntm{3 \cdot B} \leq C (3m)^d \leq C^2 3^d \cntm{B}, 
\]
hence
\[
\cntm{S_n} \leq \cntm{\bigcup\limits_{B \in \calC_n} 3 \cdot B} \leq
\sum\limits_{B \in \calC_n} c \cntm{B}
\]
and so
\[
\cntm{ \bigsqcup\limits_{B \in \calC_n} B} \geq \frac 1 c \cntm{S_n}. 
\]
Let $U_n:=\bigsqcup\limits_{B \in \calC_n} B$. The computation above
shows that 
\begin{equation}
\label{eq.stepa}
U_n \text{ covers at least } \frac 1 c \text{-fraction of } S_n
\end{equation}
and
\begin{equation}
\label{eq.stepb}
\cntm{S_1} - \cntm{U_n} \leq \cntm{S_1} - \frac 1 c \cntm{S_1} =
\frac{c-1} c \cntm{S_1}.
\end{equation}

We proceed by restricting to $(n-1)$-level balls. Assume for the moment that the
following claim is true.

\begin{claimn}
If a ball $A_{n-1}(p)$ has a nonempty intersection with $U_n$, then
$A_{n-1}(p)$ is contained in the $s$-enlargement of the ball in $\calC_n$ that it intersects.
\end{claimn}
\noindent Let
\begin{align*}
\widetilde \calC_{n-1}:=\{ A_{n-1}(p): \ &A_{n-1}(p) \text{ is a
                                           maximal } (n-1)-\text{level
                                           ball} \\
&\text{ such that } A_{n-1}(p) \cap U_n = \varnothing\}
\end{align*}
be the collection of all maximal $(n-1)$-level balls disjoint from
$U_n$ and let $\widetilde U_{n-1}$ be its union. We
apply Lemma \ref{l.fvc} once again to obtain a collection $\calC_{n-1}
\subseteq \widetilde \calC_{n-1}$ of pairwise disjoint maximal balls such that
\[
\cntm{ \bigsqcup\limits_{B \in \calC_{n-1}} B} \geq \frac 1 c
\cntm{\widetilde U_{n-1}}.
\]
Let $U_{n-1}:=\bigsqcup\limits_{B \in \calC_{n-1}} B$. In order to show that
\begin{equation}
\label{eq.s1est}
\cntm{S_1} - \cntm{ \bigcup\limits_{B \in \calC_n} \left(s  \cdot B
  \right) \cup U_{n-1}} \leq \left( \frac{c-1} c \right)^2 \cntm{S_1}
\end{equation}
it suffices to prove that
\begin{equation}
\label{eq.snm1}
\cntm{ \bigcup\limits_{B \in \calC_n} \left(s  \cdot B
  \right) \cup U_{n-1}} \geq \cntm{U_n} + \frac 1 c \cntm{S_{n-1}
  \setminus U_n},
\end{equation}
due to the obvious inequalities
\begin{align*}
\cntm{S_{n-1} \setminus U_n} \geq &\cntm{S_{n-1}} - \cntm{U_n} \geq
                               \cntm{S_1} - \cntm{U_n},\\
&\cntm{U_n} \geq \frac 1 c \cntm{S_1}.
\end{align*}
We decompose the set $S_{n-1} \setminus U_n$ as follows
\[
S_{n-1} \setminus U_n = \widetilde U_{n-1} \sqcup \left( S_{n-1}
  \setminus (U_n \cup \widetilde U_{n-1})\right).
\]
The part $S_{n-1}
  \setminus (U_n \cup \widetilde U_{n-1})$ is covered by the $(n-1)$-level balls
  intersecting $U_n$. Hence, if Claim 1 above is true, the set $S_{n-1}
  \setminus (U_n \cup \widetilde U_{n-1})$ is covered by the $s$-enlargements of
  balls in $\calC_n$. Next, $U_{n-1}$ covers at least $\frac 1 c$
  fraction of $\widetilde U_{n-1}$. It follows that the set $\bigcup\limits_{B \in \calC_n} \left(s  \cdot B
  \right) \cup U_{n-1}$ covers the set $U_n$ and at least $\frac 1
  c$-fraction of the set $S_{n-1} \setminus U_n$. Thus we have proved inequalities \eqref{eq.snm1} and \eqref{eq.s1est}. A similar
  argument shows that 
\begin{align}
\label{eq.2ndstepa}
\bigcup\limits_{B \in \calC_n} \left(s  \cdot B
  \right) \cup & \bigcup\limits_{B \in \calC_{n-1}} \left(s  \cdot B
  \right) \cup (S_n \setminus S_{n-1}) \text{ covers all but}
\\ &\text{ at most } \left( 1 - \frac 1 c\right)^2 \text{ of } S_n. \nonumber
\end{align}
Comparing Equations \eqref{eq.2ndstepa} and \eqref{eq.s1est} to
the statements $\mathrm{(a)}$ and $\mathrm{(b)}$ of the theorem, we see that the proof would be complete apart from Claim 1 if $n$ was equal to $2$.

So we proceed further to $(n-2)$-level balls and use the following claim.
\begin{claimn}
If a ball $A_{n-2}(p)$ has a nonempty
intersection with $U_n \cup U_{n-1}$, then
$A_{n-2}(p)$ is contained in the $s$-enlargement of the ball in
$\calC_n \cup \calC_{n-1}$ that it intersects. 
\end{claimn}
We let $\calC_{n-2}$ be the collection of all maximal $(n-2)$-level balls disjoint from
$U_n \cup U_{n-1}$ and let $\widetilde U_{n-2}$ be its union. We
apply Lemma \ref{l.fvc} once again to obtain a collection $\calC_{n-2}
\subseteq \widetilde \calC_{n-2}$ of pairwise disjoint balls such that
\[
\cntm{ \bigsqcup\limits_{B \in \calC_{n-2}} B} \geq \frac 1 c
\cntm{\widetilde U_{n-2}}
\] 
and let $U_{n-2}:=\bigsqcup\limits_{B \in \calC_{n-2}} B$. Similar arguments show that
\begin{equation*}
\cntm{S_1} - \cntm{ \bigcup\limits_{B \in \calC_n} \left(s  \cdot B
  \right) \cup \bigcup\limits_{B \in \calC_{n-1}} \left(s  \cdot B
  \right) \cup U_{n-2}} \leq \left( \frac{c-1} c \right)^3 \cntm{S_1}
\end{equation*}
and that the union of $s$-enlargements of balls in $\calC_n$,
$\calC_{n-1}$ and $\calC_{n-2}$, together with $S_n \setminus S_{n-2}$, covers all but
at most $\left( 1 - \frac 1 c\right)^3$ of $S_n$.

It is obvious that one can continue in this way down to the $1$-st
level balls, using the obvious
generalization of Claim 2. This would yield a collection of maximal balls
\[
\calC:=\bigcup\limits_{i=1}^n \calC_i
\]
so that the union of $s$-enlargements of balls in $\calC$ together
with $S_n \setminus S_1$ covers all but most $\left( 1 -\frac 1
  c\right)^n$ of $S_n$ and that the measure of the union of these
$s$-enlargements is at least $\left( 1- \left( 1 - \frac 1
    c\right)^n\right)$ times the measure of $S_1$.

We conclude that the proof is complete once we prove the claims above and their
generalizations. For this it suffices to prove the following statement:
\begin{claimn}
If $1 \leq i < j \leq n$ and $A_j(q)$ is a maximal ball, then for all
$p \in X$
\[
A_i(p) \cap A_j(q) \neq \varnothing \Rightarrow A_{i}(p) \subseteq s \cdot A_j(q).
\]
\end{claimn} 
Suppose this is not the case. Let $x,y$ be the centers and $r_1,r_2$
be the radii of $A_{i}(p)$
and $A_j(q)$ respectively. Recall that $s = 1+\frac 4 {r-2}$. Since
the $s$-enlargement of $A_j(q)$ does not contain $A_i(p)$, it follows
$\frac {4 r_2} {r-2} \leq 2 r_1$, hence 
\[
r r_1 \geq 2 r_1 + 2 r_2.
\]
The intersection of $A_i(p)$ and $A_j(q)$ is nonempty, hence $d \leq
r_1 + r_2$. This implies that 
\[
r r_1 \geq d+r_1+r_2,
\]
so the $r$-enlargement of the ball $A_i(p)$ contains $A_j(q)$. Since
$r \cdot A_i(p) \subseteq A_{i+1}(p)$, we conclude that the ball
$A_j(q)$ is not maximal. Contradiction.

\end{proof}

\begin{cor}
\label{c.evccor}
Suppose that in addition to all the assumptions of Theorem \ref{t.evc}
we have
\[
\cntm{S_n} \leq (c+1) \cntm{S_1},
\]
where $c$ is the constant defined in Theorem \ref{t.evc}.
Then there is a disjoint subcollection $\calC$ of maximal balls such the
union of $\left( 1+ \frac 4 {r-2}\right)$-enlargements of balls in
$\calC$ covers at least $\left( 1-(c+1) \left( \frac{c-1}{c}\right)^n
\right)$ of $S_1$.
\end{cor}
\begin{proof}
From the proof of Theorem \ref{t.evc} it follows that one can find a
disjoint collection $\calC$ of maximal balls satisfying assertions
\textrm{(a)} and \textrm{(b)} of the theorem. The statement of the
corollary is an easy consequence of \textrm{(a)}. 
\end{proof}

As the main application we will use the corollary above in the proof
of Theorem \ref{t.expdec}. It will be essential to know that one can
ensure that the extra $\left( 1+\frac 4 {r-2}\right)$-enlargement does
change the size of the union of the balls too much. 

\begin{lemma}
\label{l.smallenl}
Let $\Gamma$ be a group of polynomial growth and $\delta \in (0,1)$ be
some constant. Then there exist integers $n_0, r_0 > 2$, depending
only on $\Gamma$ and $\delta$, such that the following assertion
holds.

If $\calC$ is a finite collection of disjoint balls with radii
greater than $n_0$, then for all $r \geq r_0$ we have
\[
\cntm{\bigsqcup\limits_{W \in \calC} W} \geq
(1-\delta) \cntm{\bigcup\limits_{W\in \calC} \left( 1+\frac 4 {r-2} \right) \cdot W}.
\] 
\end{lemma}

\noindent The proof of the lemma follows from the result of Pansu (see Equation \eqref{eq.pansu}).

\section{Fluctuations of Averages of Nonnegative Functions}
\label{s.upcrineq}

The purpose of this section is to prove the following theorem.
\begin{thm}
\label{t.expdec}
Let $\Gamma$ be a group of polynomial growth of degree $d \in \Zp$ and
let $(\alpha, \beta) \subset \Rps$ be some nonempty interval. Then
there are some constants $c_1,c_2 \in \Rps$ with $c_2<1$, which depend
only on $\Gamma$, $\alpha$
and $\beta$, such that the following assertion holds. 

For any probability space $\prX=(X, \calB, \mu)$, any
measure-preserving action of $\Gamma$ on $\prX$ and any measurable $f
\geq 0$ on $X$ we have
\[
\mu(\{ x: (\avg{g \in \bl(k)} f(g \cdot x))_{k \geq 1} \in \calF_{(\alpha,\beta)}^N  \})
< c_1 c_2^N
\]
for all $N \geq 1$.
\end{thm}

To simplify the presentation we use the adjective \textbf{universal} to talk about
constants determined by $\Gamma$ and $(\alpha,\beta)$. When a constant
$c$ is determined by $\Gamma, (\alpha, \beta)$ and a parameter
$\delta$, we say that $c$ is \textbf{$\delta$-universal}. Prior to proceeding to the proof of Theorem \ref{t.expdec}, we make
some straightforward observations.

\begin{rem}
\label{r.fluctbd} It easy to see how one can generalize the theorem
above for arbitrary functions bounded from below. If a measurable function $f$ on $X$ is greater
than $-m$ for some constant $m \in \Rp$, then
\[
\mu(\{ x: (\avg{ g \in \bl(k)} f(g \cdot x))_{k \geq 1} \in \calF_{(\alpha,\beta)}^N  \})
< \widetilde{c}_1 \widetilde{c}_2^N,
\]
where the constants $\widetilde{c}_1, \widetilde{c}_2$ are given by applying Theorem \ref{t.expdec}
to the function $f+m$ and the interval $(\alpha+m,\beta+m)$. 
\end{rem}

\begin{rem}
\label{r.abclose}
Recall that $\gamma: \Zp \to \Zp$ is a growth function of a group
$\Gamma$. Let $C \geq 1$ be a constant such that
\[
\frac 1 C r^d \leq \gamma(r) \leq C r^d \quad \text{ for all } r\in \N
\]
and let $c:=3^d C^2$. Then it suffices to prove Theorem \ref{t.expdec}
only for intervals $(\alpha,\beta)$ such that
\[
\frac{\beta}{\alpha} \leq \frac{c+1}{c}.
\]
If the interval does not satisfy this condition, we replace it with a
sufficiently small subinterval and apply Theorem \ref{t.evc}. The
importance of this observation will be apparent later.
\end{rem}

\begin{rem}
\label{r.largen}
Instead of proving the original assertion of Theorem \ref{t.expdec},
we will prove the following weaker assertion, which is clearly
sufficient to deduce Theorem \ref{t.expdec}.

\emph{There is a universal integer $\widetilde N_0 \in \N$ such that for any probability space $\prX=(X, \calB, \mu)$, any
measure-preserving action of $\Gamma$ on $\prX$ and any measurable $f
\geq 0$ on $\prX$ we have
\[
\mu(\{ x: (\avg{g \in \bl(k)} f(g \cdot x))_{k \geq 1} \in \calF_{(\alpha,\beta)}^N  \})
< c_1 c_2^N
\]
for all $N \geq \widetilde N_0$.
}\end{rem}

The upcrossing inequalities given by Theorem \ref{t.expdec} and Remark
\ref{r.fluctbd} allow for a short proof of the pointwise ergodic
theorem on $\Ell{\infty}$ for actions of groups of polynomial growth.
\begin{thm}
\label{t.polergthm}
Let $\Gamma$ be a group of polynomial growth acting on a probability
space $\prX=(X, \calB, \mu)$ by measure-preserving
transformations. Then for every $f \in \Ell{\infty}(\prX)$ the limit
\[
\lim\limits_{n \to \infty} \avg{g \in \bl(n)} f(g \cdot x) 
\]
exists almost everywhere.
\end{thm}
\begin{proof}
Let 
\[
X_0:= \{x \in X: \lim\limits_{n \to \infty} \avg{g \in \bl(n)} f(g
\cdot x) \text{ does not exist} \}
\]
be the set of the points in $\prX$ where the ergodic averages do not
converge. Let $((\alpha_i,\beta_i))_{i \geq 1}$ be a sequence of
nonempty intervals such
that each nonempty interval $(c,d) \subset \R$ contains some interval
$(a_i,b_i)$. Then it is clear that if $x \in X_0$, then there is some
interval $(a_i,b_i)$ such that the sequence of averages $\left( \avg{g \in
  \bl(n)} f(g \cdot x) \right)_{n \geq
  1}$ fluctuates over $(a_i,b_i)$ infinitely often, i.e.,
\[
X_0 \subseteq \{ x \in X: \left(\avg{g \in \bl(n)} f(g \cdot x)\right)_{n \geq
  1} \in \bigcup\limits_{i \geq 1}
\bigcap\limits_{k \geq 1} \calF_{(a_i,b_i)}^k\}.
\]
By Theorem \ref{t.expdec} and Remark
\ref{r.fluctbd} we have for every interval $(a_i,b_i)$ that
\[
\mu(\{ x \in X: \left(\avg{g \in \bl(n)} f(g \cdot x)\right)_{n \geq
  1} \in \bigcap\limits_{k \geq 1} \calF_{(a_i,b_i)}^k\}) = 0,
\]
hence $\mu(X_0) = 0$ and the proof is complete.
\end{proof}

We now begin the proof of Theorem \ref{t.expdec}, namely we will prove
the assertion in Remark \ref{r.largen}. Assume from
now on that the group $\Gamma$ of polynomial growth of degree $d \in
\Zp$ and the interval $(\alpha, \beta) \subset \Rps$ are
\emph{fixed}. 

Given a measure-preserving action of $\Gamma$ on a probability space
$\prX =(X,\calB,\mu)$, let
\[
E_N:=\{ x: \left( \avg{g \in \bl(k)} f(g \cdot x) \right)_{k \geq 1} \in
\calF_{(\alpha, \beta)}^N\}
\]
be the set of all points $x \in X$ where the ergodic averages
fluctuate at least $N \geq \widetilde N_0$
times across the interval $(\alpha,\beta)$. Here $\widetilde N_0$ is a universal
constant, which will be determined later. For $m \geq 1$ define,
furthermore, the set
\[
E_{N,m}:=\{ x: \left( \avg{g \in \bl(k)} f(g \cdot x) \right)_{k=1}^m \in
\calF_{(\alpha, \beta)}^N\}
\]
of all points such that the finite sequence $\left(\avg{\bl(k)} f(g
  \cdot x) \right)_{k=1}^m$ fluctuates at least $N$ times across
$(\alpha,\beta)$. Then, clearly, $(E_{N,m})_{m \geq 1}$ is a monotone
increasing sequence of sets and
\[
E_N = \bigcup\limits_{m \geq N} E_{N,m}.
\]
We will complete the proof by giving a universal estimate for
$\mu(E_{N,m})$ for all $m \geq N$. For that we use the
transference principle (Lemma \ref{l.caldtrans}), i.e., for an integer
$L > m$ and a point
$x \in X$ we let
\[
B_{L,m,x}:=\{g: \ g \cdot x \in E_{N,m} \text{  and  } \| g \| \leq L-m \}.
\]
The goal is to show that the density of the set
\[
B_0:=B_{L,m,x} \subset \bl(L)
\]
can be estimated by $c_1 c_2^N$ for some
universal constants $c_1,c_2$. The main idea is as follows. For every point $z
\in B_0$ the sequence of averages
\[
k \mapsto \avg{g \in \bl(k)} f( (gz) \cdot x), \quad k = 1,\dots,m
\]
fluctuates at least $N$ times. Since the word metric $d=d_R$ on $\Gamma$ is
right-invariant, the set $\bl(k)z$ is in fact a ball of radius
$k$ centered at $z$ for each $k=1,\dots,m$. Given a parameter $\delta
\in (0, 1-\sqrt{{\alpha}/{\beta}})$, we will pick some of these balls
and apply effective Vitali covering theorem
(Theorem \ref{t.evc}) multiple times to replace $B_0$ by a sequence 
\[
B_1,B_2,\dots, B_{\lfloor (N-N_0)/T \rfloor}
\]
of subsets of $\bl(L)$ for some $\delta$-universal integers $T, N_0 \in \N$ which satisfies the
assumption
\begin{equation}
\label{eq.oddstep}
B_{2i+1} \text{ covers at least } \left( 1-\delta \right)-\text{fraction of } B_{2i} \quad \text{ for all
  indices } i \geq 0
\end{equation}
at `odd' steps and the assumption
\begin{equation}
\label{eq.evenstep}
\cntm{B_{2i}} \geq \frac{\beta}{\alpha}(1-\delta) \cntm{B_{2i-1}}
\quad \text{ for all indexes } i \geq 1
\end{equation}
at `even' steps. Each $B_i$ is, furthermore, a union 
\[
\bigsqcup\limits_{B \in \calC_i} B
\]
of some family $\calC_i$ of disjoint balls with centers in $B_0$. If
such a sequence of sets $B_1,\dots,B_{\lfloor (N-N_0)/T
  \rfloor}$ exists, then
\begin{align*}
\cntm{\bl(L)} \geq \cntm{B_{\lfloor (N-N_0)/T \rfloor}} \geq \left( \frac{\beta}{\alpha}(1-\delta)^2
  \right)^{\lfloor \frac {N-N_0}{2 T} \rfloor } \cntm{B_0},
\end{align*}
which gives the required exponential bound on the density of $B_0$
with
\[
c_2:=\left( \frac{\alpha}{\beta}(1-\delta)^{-2}
  \right)^{ 1 / 2T}
\]
and a suitable $\delta$-universal $c_1$. To
ensure that conditions \eqref{eq.oddstep} and \eqref{eq.evenstep}
hold, one has to pick sufficiently large $\delta$-universal parameters
$r$ and $n$ for the effective Vitali covering theorem. We make it precise at
the end of the proof, for now we assume that $r$, $n$ are `large enough'. 

In order to force the
sufficient growth rate of the balls (condition \textrm{(b)} of Theorem
\ref{t.evc}), we employ the following
argument. Let $K>0$ be the smallest integer such that
\[
\left(1-\frac{1-(\alpha/\beta)^{1/d}}{2} \right)^{\lceil \frac K 2 \rceil} \left(
  \frac{\beta}{\alpha} \right)^{{\lceil \frac K 2 \rceil} \cdot \frac
  1 d} \geq r.
\]
Then, applying Corollary \ref{c.skip}, we obtain a universal integer
$n_0 \in \N$ such that if a sequence
\[
(\avg{g \in B(i)} f((gz) \cdot x))_{i=n}^m \quad \text{ for some } n>n_0, z \in B_0
\]
fluctuates at least $K$ times across the interval $(\alpha,\beta)$, then 
\begin{equation}
\label{eq.evccondb}
\frac m n > \left(1-\frac{1-(\alpha/\beta)^{1/d}}{2}\right)^{\lceil \frac K 2 \rceil} \left(
  \frac{\beta}{\alpha} \right)^{{\lceil \frac K 2 \rceil} \cdot \frac
  1 d} \geq r.
\end{equation}
Let $n$ be large enough for use in effective Vitali covering
theorem. We define $T:=2nK$ and let $N_0 \geq n_0$ be sufficiently large (this will be made precise later). The first $N_0$
fluctuations are skipped to ensure that the balls have large enough radius, and
the rest are divided into $\lfloor (N-N_0) / T \rfloor$ groups of $T$
consecutive fluctuations. The $i$-th group of consecutive
fluctuations is used to construct the set $B_i$ for $i=1,\dots,\lfloor
(N-N_0)/T \rfloor$ as follows. We distinguish between the `odd' and the
`even' steps.

\noindent
\textbf{Odd step:} First, let us describe the procedure for odd
$i$'s. For each point $z \in B_{i-1}$ we do the following.
By induction we assume that $z \in B_{i-1}$ belongs to some
unique ball $\bl(u,s)$ from $(i-1)$-th step
with $u \in B_0$. If $i=1$, then $z \in B_0$. Let $A_1(z)$ be the $(K+1)$-th ball $\bl(u,s_1)$ in the
$i$-th group of fluctuations such that
\[
\avg{g \in A_1(z)} f(g \cdot x) > \beta,
\]
$A_2(z)$ be the $(2K+1)$-th ball $\bl(u,s_2)$ in the $i$-th group of
fluctuations such that
\[
\avg{g \in A_2(z)} f(g \cdot x) > \beta
\]
and so on up to $A_n(z)$. It is clear that the $r$-enlargement of
$A_j(z)$ is contained in $A_{j+1}(z)$ for all indexes $j<n$ and that
the balls defined in this manner are contained in $\bl(L)$. Thus the
assumptions of Theorem \ref{t.evc} are satisfied. There are two
further possibilities: either this collection satisfies the
additional assumption in Corollary \ref{c.evccor}, i.e., 
\begin{equation}
\label{eq.corcond}
\cntm{S_n} \leq (c+1) \cntm{S_1}
\end{equation}
or not. If \eqref{eq.corcond} holds, then by the virtue of
Corollary \ref{c.evccor} we obtain a disjoint
collection $\calC$ of maximal balls such that the measure of the union
of $\left( 1+\frac 4 {r-2} \right)$-enlargements of balls in $\calC$
covers at least $\left( 1 - (c+1) \left(\frac{c-1}{c} \right)^n \right)$
of $S_1$. We let
\[
B_{i}:=\bigsqcup\limits_{B \in \calC} B
\]
and $\calC_i:=\calC$. Condition \eqref{eq.oddstep} is satisfied if $r$
and $n$ are large enough, and we proceed to the
following `even' step. If, on the
contrary,
\[
\cntm{S_n} > (c+1) \cntm{S_1},
\]
then we apply the standard Vitali covering lemma to the collection of
maximal $n$-th level balls and obtain a disjoint subcollection $\calC$
such that
\begin{equation}
\label{eq.10cincr}
\cntm{\bigsqcup\limits_{B \in \calC} B} \geq \frac{1}{c}\cntm{S_n} > \frac{c+1}{c}\cntm{S_1}
\end{equation}
We assume without loss of generality that $\frac{\beta}{\alpha} \leq
\frac{c+1}{c}$ (see Remark \ref{r.abclose}). We let
\begin{align*}
B_i&:=B_{i-1}, \\
B_{i+1}&:=\bigsqcup\limits_{B \in \calC} B
\end{align*}
and
\begin{align*}
\calC_i&:=\calC_{i-1}, \\
\calC_{i+1}&:=\calC.
\end{align*}
The conditions \eqref{eq.oddstep}, \eqref{eq.evenstep} are satisfied
and we proceed to the next `odd' step.

\noindent
\textbf{Even step:} We now describe the procedure for even
$i$'s. For each point $z \in B_{i-1}$ we do the following.
By induction we assume that $z \in B_{i-1}$ belongs to some
unique ball $\bl(u,s)$ from $(i-1)$-th step
with $u \in B_0$. Let $A_1(z)$ be the $(K+1)$-th ball $\bl(u,s_1)$ in the
$i$-th group of fluctuations such that
\[
\avg{g \in A_1(z)} f(g \cdot x) < \alpha,
\]
$A_2(z)$ be the $(2K+1)$-th ball $\bl(u,s_2)$ in the $i$-th group of
fluctuations such that
\[
\avg{g \in A_2(z)} f(g \cdot x) < \alpha
\]
and so on up to $A_n(z)$. It is clear that the $r$-enlargement of
$A_j(z)$ is contained in $A_{j+1}(z)$ for all indexes $j<n$ and that
the balls defined in this manner are contained in $\bl(L)$. Thus the
assumptions of Theorem \ref{t.evc} are satisfied. There are two
further possibilities: either this collection satisfies the
additional assumption in Corollary \ref{c.evccor}, i.e., 
\begin{equation}
\label{eq.corcond1}
\cntm{S_n} \leq (c+1) \cntm{S_1}
\end{equation}
or not. If 
\[
\cntm{S_n} > (c+1) \cntm{S_1},
\]
then we apply the standard Vitali covering lemma to the collection of
maximal $n$-th level balls and obtain a disjoint subcollection $\calC$
such that
\begin{equation}
\label{eq.10cincr1}
\cntm{\bigsqcup\limits_{B \in \calC} B} \geq \frac{1}{c}\cntm{S_n} > \frac{c+1}{c}\cntm{S_1}
\end{equation}
We assume without loss of generality that $\frac{\beta}{\alpha} \leq
\frac{c+1}{c}$ (see Remark \ref{r.abclose}). We let
\[
B_i:=\bigsqcup\limits_{B \in \calC} B
\]
and proceed to the following `odd' step. If \eqref{eq.corcond1} holds, then by the virtue of
Corollary \ref{c.evccor} we obtain a disjoint
collection $\calC$ of maximal balls such that the measure of the union
of $\left( 1+\frac 4 {r-2} \right)$-enlargements of balls in $\calC$
covers at least $\left( 1 - (c+1) \left(\frac{c-1}{c} \right)^n \right)$
of $S_1$. We let
\[
B_{i}:=\bigsqcup\limits_{B \in \calC} B
\]
and $\calC_i:=\calC$. The goal is to prove that condition \eqref{eq.evenstep} is
satisfied. If the balls from $\calC_{i-1}$ were completely contained in
the balls from $\calC_i$, the proof would be completed by applying
Lemma \ref{l.ballgrowth}. This, in general, might not be the case, so
we argue as follows. First, we prove the following lemma.

\begin{lemma}
\label{l.bdrint}
If a ball $W_1$ from $\calC_{i-1}$ intersects $\intr{r}(W_2)$
for some ball $W_2 \in \calC_i$, then $W_1 \subseteq W_2$.
\end{lemma}
\begin{proof}
Let $W_1 = \bl(y_1,s_1)$ and $W_2=\bl(y_2,s_2)$ for some $y_1,y_2 \in B_0$. Since $W_1$
intersects $\intr{r}(W_2)$, we have
\[
d(y_1,y_2) \leq s_2(1-5/r)+s_1.
\]
If $W_1$ is not contained in $W_2$, then $d(y_1,y_2) > s_2-s_1$. From
these inequalities it follows that
\[
s_1 \geq d(y_1,y_2)-s_2(1-5/r) >s_2-s_1-s_2+\frac{5s_2}{r},
\]
hence $s_2<\frac{2rs_1}{5}$. We deduce that the $r$-enlargement of
$W_1$ contains $W_2$. This is a contradiction since $W_2$ is maximal
and the $r$-enlargement of $W_1$ is contained in $n$-th level ball $A_n(y_1)$.
\end{proof}

From the lemma above it follows that the set $B_{i-1}$ can be
decomposed as 
\begin{align*}
B_{i-1} = \left( \bigsqcup\limits_{W \in \calC_{i-1}'} W \right) \sqcup (\bdr{r}(\calC_i) \cap
  B_{i-1}) \sqcup (B_{i-1} \setminus B_{i}), 
\end{align*}
where 
\[
\calC_{i-1}':=\{ W \in \calC_{i-1}: \ W \cap \intr{r}(V) \neq \varnothing
\text{ for some } V \in \calC_i \}.
\]
The rest of the argument depends on how much of $B_{i-1}$ is contained
in $\bdr{r}(\calC_i)$, so let
\[
\Delta:=\frac{\cntm{\bdr{r}(\calC_i) \cap B_{i-1}}}{\cntm{B_{i-1}}}.
\]
There are two possibilities. First, suppose that $\Delta>\frac{\delta}
3$. Then $\cntm{B_{i-1}} \leq
\frac{\cntm{\bdr{r}(\calC_i)}}{\delta/3}$. Let $r$ and the radii of the
balls in $\calC_i$ be large enough
(see Lemma \ref{l.smallbdr}) so
that
\[
\frac{\cntm{\bdr{r}(\calC_i)}}{\cntm{B_i}} <
\frac{\alpha}{\beta} \frac{\delta} 3 (1-\delta)^{-1}.
\]
It is then easy to see that condition \eqref{eq.evenstep} is
satisfied. Suppose, on the other hand,
that $\Delta \leq \frac{\delta} 3$. Then, if $n$ and $r$ are large enough so
that $\cntm{B_{i-1} \setminus B_i}$ is small compared to
$\cntm{B_{i-1}}$, we obtain
\begin{align*}
\cntm{B_{i-1}} &\leq
                 \frac{\alpha}{\beta}\cntm{B_i}+\cntm{\bdr{r}(\calC_i)
                 \cap B_{i-1}}+\cntm{B_{i-1} \setminus B_i} \leq \\
&\leq \frac{\alpha}{\beta}\cntm{B_i}+\frac{\delta}{3}
\cntm{B_{i-1}}+\frac{\delta} 3 \cntm{B_{i-1}},
\end{align*}
which implies that
\[
\cntm{B_i} \geq \frac{\beta}{\alpha}(1-\frac{2 \delta} 3) \cntm{B_{i-1}},
\]
i.e., condition \eqref{eq.evenstep} is satisfied as well. We proceed to the
following `odd' step. 

The proof of the theorem is essentially complete. To finish it we only
need to say how one can choose the constants $N_0, r, n$ and
$\widetilde N_0$. Recall that $\delta \in
(0, 1- \left( \alpha / \beta \right)^{1/2})$ is an arbitrary parameter. 
First, the integer $n \in \N$ is chosen so that
\[
(c+1) \left( 1- \frac 1 c\right) ^n \leq 1-\sqrt{1-\delta / 4}.
\]
Next, we
choose $r$ as the maximum of
\begin{aufziii}
\item the integer $r_0$ given by Lemma \ref{l.smallbdr} with
  the parameter $\frac{\alpha}{\beta}\frac{\delta} 3 (1-\delta)^{-1}$;


\item the integer $r_0$ given by Lemma \ref{l.smallenl} with the
  parameter $1-\sqrt{1-\delta/4}$.
\end{aufziii}
The integer $K>0$ is picked so that condition \eqref{eq.evccondb} is
satisfied. We choose $N_0$ as the maximum of
\begin{aufziii}
\item the integer $n_0$ given by Lemma \ref{l.smallbdr} with
  the parameter $\frac{\alpha}{\beta}\frac{\delta} 3 (1-\delta)^{-1}$;

\item the integer $n_0$ given by Lemma \ref{l.smallenl} with the
  parameter $1-\sqrt{1-\delta/4}$;

\item the integer $n_0$ given by Corollary \ref{c.skip} with the
  parameter $\frac{1-(\alpha/\beta)^{1/d}}{2}$;
\end{aufziii}
Finally, we define $\widetilde N_0$ as $\widetilde N_0:=N_0+4nK+1$. A straightforward computation shows that this choice of constants
satisfies all requirements. We do not assert, however, that this
choice yields \emph{optimal} constants $c_1$ and $c_2$. \qed

\printbibliography[]
\end{document}

%% file: makro.tex
\newcommand{\vanish}[1]{\relax}       
\def\qedsymbol{\hbox to 1ex{\llap{\rule{0.25pt}{1ex}}\rlap{\rule{1ex}{0.25pt}}\lower0.25pt\rlap{\raise1ex\rlap{\rule{1ex}{0.25pt}}}\hskip1ex\llap{\rule{0.25pt}{1ex}}}}
\def\rlqed{\rlap{\rule{\hsize}{0pt}\kern-1ex\kern-1em\qed}} 

\newcounter{aufzi}
\newenvironment{aufzi}{\begin{list}{ {\upshape(\alph{aufzi})}}{
        \usecounter{aufzi}
        \topsep1ex
        \parsep0cm
        \itemsep0.8ex
        \leftmargin1cm
        \labelwidth0.5cm
        \labelsep0.3cm
}}
{\end{list}}

\newcounter{aufzii}

\newcounter{aufziii}
\newenvironment{aufziii}{\begin{list}{ {\upshape\arabic{aufziii})}}{
        \usecounter{aufziii}
        \topsep1ex
        \parsep0cm
        \itemsep0.8ex
        \leftmargin1cm
        \labelwidth0.5cm
        \labelsep0.3cm
}}
{\end{list}}



\newcommand{\bbE}{\mathbb{E}}

\newcommand{\calB}{\mathcal{B}}
\newcommand{\calC}{\mathcal{C}}

\newcommand{\calF}{\mathcal{F}}

\def\ue{\mathrm{e}}

\renewcommand{\ue}{\mathrm{e}}    

\newcommand{\R}{\mathbb{R}}     
\newcommand{\N}{\mathbb{N}}
\newcommand{\Z}{\mathbb{Z}}

\newcommand{\sdif}{\triangle}      

\newcommand{\Bigcap}[2][\relax]{%
 \ifx#1\relax \bigcap_{#2}
 \else \bigcap^{#1}_{#2}
 \fi}
\newcommand{\Bigcup}[2][\relax]{%
 \ifx#1\relax \bigcup_{#2}
 \else \bigcup^{#1}_{#2}
 \fi}

\def\fact#1#2{#1/#2}
\def\tfact#1#2{#1/#2}

\def\fact#1#2{{\raise0.2em\hbox{$#1$}\kern-0.2em/\kern-0.1em\lower0.2em\hbox{$#2$}}}
\def\tfact#1#2{{\raise0.1em\hbox{\small$#1$}\kern-0.1em/\kern-0.1em\lower0.1em\hbox{\small$#2$}}}

\newcommand{\norm}[2][\relax]{
   \ifx#1\relax \ensuremath{\left\Vert#2\right\Vert}
   \else \ensuremath{\left\Vert#2\right\Vert_{#1}}
   \fi}

\newcommand{\Bnorm}[2][\relax]{
   \ifx#1\relax \ensuremath{\Bigl\Vert#2\Bigr\Vert}
   \else \ensuremath{\Bigl\Vert#2\Bigr\Vert_{#1}}
   \fi}

\makeatletter
\newcommand{\tdprod}[2]{\ensuremath{%
  \setbox0=\hbox{\ensuremath{\langle#1,#2 \rangle}}
  \dimen@\ht0
  \advance\dimen@ by \dp0 (#1\rule[-\dp0]{0pt}{\dimen@}\,|#2\hspace{1pt})}}
\newcommand{\dprod}[2]{\ensuremath{%
  \setbox0=\hbox{\ensuremath{\left\langle#1,#2\right\rangle}}
  \dimen@\ht0
  \advance\dimen@ by \dp0 \left\langle\left.#1\rule[-\dp0]{0pt}{\dimen@}\,\right|#2\hspace{1pt}\right\rangle}}

\newcommand{\bdprod}[2]{\ensuremath{%
  \setbox0=\hbox{\ensuremath{\bigl\langle#1,#2\bigr\rangle}}
  \dimen@\ht0
  \advance\dimen@ by \dp0 \bigl\langle#1\bigl|\rule[-\dp0]{0pt}{\dimen@}\bigr.#2\hspace{1pt}\bigr\rangle}}
\newcommand{\Bdprod}[2]{\ensuremath{%
  \setbox0=\hbox{\ensuremath{\Bigl\langle#1,#2\Bigr\rangle}}
  \dimen@\ht0
  \advance\dimen@ by \dp0 \Bigl\langle#1\Bigl|\rule[-\dp0]{0pt}{\dimen@}\Bigr.#2\hspace{1pt}\Bigr\rangle}}
\newcommand{\tsprod}[2]{\ensuremath{%
  \setbox0=\hbox{\ensuremath{(#1,#2)}}
  \dimen@\ht0
  \advance\dimen@ by \dp0 (#1\rule[-\dp0]{0pt}{\dimen@}\,|#2\hspace{1pt})}}
\newcommand{\sprod}[2]{\ensuremath{%
  \setbox0=\hbox{\ensuremath{\left(#1,#2\right)}}
  \dimen@\ht0
  \advance\dimen@ by \dp0 \left(\left.#1\rule[-\dp0]{0pt}{\dimen@}\,\right|#2\hspace{1pt}\right)}}
\newcommand{\bsprod}[2]{\ensuremath{%
  \setbox0=\hbox{\ensuremath{\bigl(#1,#2\bigr)}}
  \dimen@\ht0
  \advance\dimen@ by \dp0 \bigl(#1\bigl|\rule[-\dp0]{0pt}{\dimen@}\bigr.#2\hspace{1pt}\bigr)}}
\newcommand{\Bsprod}[2]{\ensuremath{%
  \setbox0=\hbox{\ensuremath{\Bigl(#1,#2\Bigr)}}
  \dimen@\ht0
  \advance\dimen@ by \dp0 \Bigl(#1\Bigl|\rule[-\dp0]{0pt}{\dimen@}\Bigr.#2\hspace{1pt}\Bigr)}}
\makeatother

\newcommand{\Ell}[2][\relax]{
   \ifx#1\relax \mathrm{L}^{\mathrm{#2}}
   \else \mathrm{L}^{\mathrm{#2}}_{\mathrm{#1}}
   \fi}
\renewcommand{\Ell}[2][\relax]{
   \ifx#1\relax \mathrm{L}^{\!#2}
   \else \mathrm{L}^{\!#2}_{\mathrm{#1}}
   \fi}

\newcommand{\Wee}[2][\relax]{
   \ifx#1\relax \mathrm{W}^{\mathrm{#2}}
   \else \mathrm{W}^{\mathrm{#2}}_{\mathrm{#1}}
   \fi}
\newcommand{\Har}[2][\relax]{
   \ifx#1\relax \mathsf{H}^{\mathsf{#2}}
   \else   \mathsf{H}^{\mathsf{#2}}_{\mathrm{#1}}
   \fi}

\def\prX{\mathrm X}    

\def\rlqed{\rlap{\rule{\hsize}{0pt}\kern-1ex\kern-1em\qed}}

\makeatletter

\def\maketag@@@@@#1{\llap{\hbox to\hsize{\m@th\normalfont#1}}%
\gdef\tagform@##1{\maketag@@@{(\ignorespaces##1\unskip\@@italiccorr)}}}

\def\eqtext#1{\gdef\tagform@##1{\maketag@@@@@{\ignorespaces##1\unskip\@@italiccorr\hfill}}\tag{#1}}%
\def\reqtext#1{\gdef\tagform@##1{\maketag@@@@@{\hfill\ignorespaces##1\unskip\@@italiccorr}}\tag{#1}}%
\def\leqtext#1{\gdef\tagform@##1{\maketag@@@@@{\ignorespaces##1\unskip\@@italiccorr}}\tag{#1}}%
\makeatother

